\documentclass{amsart}
\usepackage{graphicx}
\usepackage{amsmath}
\usepackage{amssymb}
\usepackage[all]{xy}

\newcommand{\Hom}{\operatorname{Hom}\nolimits}
\renewcommand{\Im}{\operatorname{Im}\nolimits}

\newcommand{\Coker}{\operatorname{Coker}\nolimits}

\newcommand{\pd}{\operatorname{pd}\nolimits}

\newcommand{\CIdim}{\operatorname{CI-dim}\nolimits}
\newcommand{\Gdim}{\operatorname{G-dim}\nolimits}

\newcommand{\depth}{\operatorname{depth}\nolimits}
\newcommand{\Ann}{\operatorname{Ann}\nolimits}

\newcommand{\Ext}{\operatorname{Ext}\nolimits}
\newcommand{\Tateext}{\operatorname{\widehat{Ext}}\nolimits}

\newcommand{\Id}{\operatorname{Id}\nolimits}
\renewcommand{\H}{\operatorname{H}\nolimits}

\newcommand{\rank}{\operatorname{rank}\nolimits}
\newcommand{\m}{\operatorname{\mathfrak{m}}\nolimits}
\newcommand{\n}{\operatorname{\mathfrak{n}}\nolimits}

\newcommand{\cx}{\operatorname{cx}\nolimits}

\newcommand{\s}{\operatorname{\Sigma}\nolimits}
\newcommand{\Noethfl}{\operatorname{Noeth^{\mathrm{fl}}}\nolimits}

\newtheorem{theorem}{Theorem}[section]

\newtheorem{corollary}[theorem]{Corollary}

\newtheorem{lemma}[theorem]{Lemma}
\newtheorem{proposition}[theorem]{Proposition}
\theoremstyle{definition}
\newtheorem*{definition}{Definition}
\theoremstyle{definition}
\newtheorem{question}[theorem]{Question}
\newtheorem{construction}[theorem]{Construction}
\theoremstyle{definition}

\theoremstyle{definition}

\theoremstyle{definition}

\theoremstyle{definition}

\theoremstyle{remark}

\theoremstyle{definition}

\theoremstyle{definition}

\begin{document}

\title{On growth in totally acyclic minimal complexes}
\author{Petter Andreas Bergh \& David A.\ Jorgensen}

\address{Petter Andreas Bergh \\ Institutt for matematiske fag \\
  NTNU \\ N-7491 Trondheim \\ Norway}
\email{bergh@math.ntnu.no}

\address{David A.\ Jorgensen \\ Department of mathematics \\ University
of Texas at Arlington \\ Arlington \\ TX 76019 \\ USA}
\email{djorgens@uta.edu}




\begin{abstract}
Given a commutative Noetherian local ring, we provide a criterion
under which a totally acyclic minimal complex of free modules has
symmetric growth.
\end{abstract}

\maketitle

\section{Introduction}\label{intro}

Given a commutative Noetherian local ring, when does a totally
acyclic minimal complex of free modules have symmetric growth? In
other words, given such a complex, does the left growth of the ranks
of its free modules equal the right growth? Avramov and Buchweitz
show in \cite{AvramovBuchweitz} that this is always the case for
totally acyclic minimal complexes of free modules over local
complete intersections. However, Jorgensen and \c{S}ega showed in
\cite{JorgensenSega} that it does not hold for a local ring in
general, even when the ring is Gorenstein. In fact, they constructed
such a ring and a totally acyclic minimal complex whose left growth
is exponential and right growth is constant. (The characteristics
of growth in the dual complex are thus reversed.)

Totally acyclic complexes can be used to compute Tate, or stable
(co)homology.  Since characteristics of the underlying complex
are often reflected in characteristics of the derived (co)homology,
it is of interest to study general properties of the underlying
complexes.

In this paper, we give a criterion under which symmetric growth
of totally acyclic minimal complexes
holds. This criterion is given in terms of the cohomology of the
image of a given differential in the complex. Namely, we show that
if the cohomology is finitely generated with respect to a ring
acting centrally on the derived category, and the ring action commutes
with dualization, then the complex has
symmetric polynomial growth.  As a corollary of our main theorem we
prove that whenever an image in the complex and its dual have complete
intersection dimension zero, then the complex has symmetric polynomial
growth.  Since all such images and their duals have complete intersection
dimension zero when the ring is a local complete intersection, we 
recover the result of Avramov and Buchweitz cited above.

\section{Notation and terminology}\label{notation}

Throughout this section, we fix a local (meaning commutative
Noetherian local) ring $(A, \m,k)$ and a finitely generated
$A$-module $M$. All modules we encounter are assumed to be
finitely generated. We denote by $M^*$ the $A$-dual of $M$, that
is, the $A$-module $\Hom_A(M,A)$. If the canonical homomorphism $M
\to M^{**}$ is an isomorphism, then $M$ is said to be
\emph{reflexive}. Furthermore, if $M$ is reflexive, and
$$\Ext_A^n(M,A) = 0 = \Ext_A^n(M^*,A)$$
for $n \ge 1$, then $M$ is a module of \emph{Gorenstein dimension
zero} (alternatively, $M$ is said to be \emph{totally reflexive}).
We shall write ``G-dimension" instead of ``Gorenstein dimension".

Suppose now that $M$ has G-dimension zero, and fix free
resolutions
$$
\cdots \to C_2 \to C_1 \to C_0 \xrightarrow{d_0} M \to 0
$$
$$
\cdots \to C_{-3} \to C_{-2} \to C_{-1} \to M^* \to 0
$$
of $M$ and $M^*$, respectively. Dualizing the latter resolution,
we obtain a complex
$$
0 \to M \to C_{-1} \to C_{-2} \to C_{-3} \to \cdots
$$
which is exact since $\Ext_A^n(M^*,A)=0$ for $n \ge 1$. Splicing
this sequence with the free resolution of $M$, we obtain a doubly
infinite exact sequence
$$
C \colon \cdots \to C_2 \to C_1 \to C_0
\xrightarrow{d_0} C_{-1} \to C_{-2} \to \cdots
$$
of free modules, in
which $M \cong \Im d_0$. The ``left part" of the dualized complex
$\Hom_A(C,A)$ is exact, since it is just the free
resolution of $M^*$. Moreover, the ``right part" is also exact,
since $\Ext_A^n(M,A)=0$ for $n \ge 1$. Consequently, the complex
$\Hom_A(C,A)$ is exact; thus $C$ is \emph{totally
acyclic}. Conversely, if $C$ is a totally acyclic complex
of free modules, then the image of any of its differentials has
G-dimension zero. Given such a complex, we denote by
$M_C$ the image of the zeroth differential.

When $M$ has G-dimension zero, then the totally acyclic complex
constructed above is a \emph{complete resolution} of $M$. By
\cite{Buchweitz}, \cite{CornickKropholler}, it is unique up to
homotopy equivalence. Consequently, for every $n \in \mathbb{Z}$ and
every $A$-module $N$, the \emph{Tate cohomology} module (or
\emph{stable cohomology} module)
$$
\Tateext_A^n(M,N) \stackrel{\text{def}}{=} \H_{-n}
\left ( \Hom_A(C,N)
\right )
$$
is independent of the choice of complete resolution of
$M$. By construction, there is an isomorphism $\Tateext_A^n(M,N)
\cong \Ext_A^n(M,N)$ whenever $n>0$. We refer to
\cite{AvramovMartsinkovsky} for general properties of Tate
cohomology modules.

When we start with minimal free resolutions of $M$ and $M^*$,
we end up with an \emph{almost minimal} totally acyclic complex $C$,
meaning that the image of the differential $C_{n+1} \to C_n$ is
contained in $\m C_n$ for all $n\ne -1$. The complex $C$ will be
\emph{minimal} if in addition the image of $d_0$ lies in $\m C_{-1}$
(and this happens precisely when $M$ has no nonzero free summands).
Since minimal free resolutions are unique up to isomorphism, the integers
$\beta_n (M) \stackrel{\text{def}}{=} \rank C_n$ are well-defined
for all $n\in \mathbb Z$.
The integer $\beta_n(M)$ is called the $n$th \emph{Betti number} of $M$,
and for $n\ne 0,-1$ it is equal to the dimension of the $k$-vector space
$\Tateext_A^n(M,k)$.
We define the \emph{positive complexity} and \emph{negative complexity} of $M$
as
\begin{eqnarray*}
\cx_A^{+} M & \stackrel{\text{def}}{=} & \inf \{ t \in \mathbb{N}
\cup \{ 0 \} \mid \exists a \in \mathbb{R} \text{ such that }
\beta_n (M) \le an^{t-1} \text{ for } n \gg 0 \}, \\
\cx_A^{-} M & \stackrel{\text{def}}{=} & \inf \{ t \in \mathbb{N}
\cup \{ 0 \} \mid \exists a \in \mathbb{R} \text{ such that }
\beta_n (M) \le a(-n)^{t-1} \text{ for } n \ll 0 \}.
\end{eqnarray*}
These measure on a polynomial scale the left and the right, respectively,
rate of growth of the minimal complete resolution of $M$. We also define the
corresponding Poincar{\'e} series
\begin{eqnarray*}
P_A^{+}(M,t) & \stackrel{\text{def}}{=} & \sum_{n\ge 0}\beta_n(M) t^n, \\
P_A^{-}(M,t) & \stackrel{\text{def}}{=} & \sum_{n\ge 0}\beta_{-n}(M) t^n,
\end{eqnarray*}
that is, the generating functions of the Betti numbers. Note that
the positive complexity coincides with the ``ordinary" complexity of
a module, that is, the polynomial rate of growth in its minimal free
resolution.

The aim of this paper is to give sufficient conditions for a totally
acyclic minimal complex of free modules to have symmetric growth. In
other words, given such a complex $C$, we give a criterion
for when $\cx_A^{+} M_C$ and $\cx_A^{-} M_C$
are equal. Namely, we show that $C$ has symmetric growth whenever
the cohomology of $M_C$ behaves well with respect to dualization and is 
``finitely generated.'' To explain these notions, consider the bounded derived
category $D^b(A)$ of finitely generated $A$-modules. This is a
triangulated category, whose suspension functor $\s$ is just the
left shift of a complex. Given complexes $X,Y$ and an integer $n$,
we denote the graded $A$-module $\Hom_{D^b(A)}(X, \s^nY)$ by
$\Ext_A^n(X,Y)$; for modules, this is just the usual $\Ext$. We
denote by $\Ext_A(X,Y)$ the graded module $\oplus_{n \in
\mathbb{Z}} \Ext_A^n(X,Y)$, and if $n_0$ is an integer then we set
$\Ext_A^{\ge n_0}(X,Y) = \oplus_{n\ge n_0}\Ext_A^n(X,Y)$. The
\emph{graded center} of $D^b(A)$, denoted $Z ( D^b(A) )$, is a
graded ring $Z ( D^b(A) ) = \oplus_{n \in \mathbb{Z}} Z^n ( D^b(A)
)$, whose degree $n$ component $Z^n( D^b(A) )$ is the set of natural
transformations $\Id \xrightarrow{f} \s^n$ satisfying $f_{\s X} =
(-1)^n \s f_X$ on the level of objects. For details and properties
of the graded center, see \cite{BuchweitzFlenner}.

Now let $H = \oplus_{n\ge 0}H_n$ be a positively graded ring
which is graded-commutative, that is, $\eta \theta = (-1)^{|\eta|
|\theta|} \theta \eta$ for all homogeneous elements $\eta, \theta
\in H$. We say that $H$ \emph{acts centrally} on $D^b(A)$ if there
exists a homomorphism $H \to Z ( D^b(A) )$ of graded rings. In this
case, for every complex $X \in D^b(A)$ there is a graded ring
homomorphism
$$H \xrightarrow{\varphi_X} \Ext_A(X,X),$$
and for every complex $Y$ and all homogeneous elements $\eta \in H,
\theta \in \Ext_A(X,Y)$ the equality $\varphi_Y( \eta ) \circ
\theta = (-1)^{|\eta| |\theta|} \theta \circ \varphi_X ( \eta )$
holds. In other words, the left and right $H$-module structures on
$\Ext_A(X,Y)$ coincide up to a sign.

\begin{definition} We say that $\Ext_A(X,Y)$
is an \emph{eventually Noetherian $H$-module of finite length}, and
write $\Ext_A(X,Y) \in \Noethfl H$, if the following holds: there
is a number $n_0$ such that the $H$-module $\Ext_A^{\ge n_0}(X,Y)$
is Noetherian and the length of  $\Ext_A^n(X,Y)$ as an $H_0$-module,
denoted by $\ell_{H_0}$, is finite for each $n \ge n_0$.
\end{definition}

Let $M$ be an $A$-module of G-dimension zero,
$C$ be a totally acyclic complex of free modules with
$M=M_C$, and $\theta$ be an element
of $\Ext_A^n(M,M)$ for some $n$. This element $\theta$
corresponds to a chain map
$C\to \s^n C$ (and equivalent elements in $\Ext_A^n(M,M)$ correspond
to homotopic chain maps).
Dualizing, we obtain a chain map $\s^{-n}(C^*)=(\s^n C)^*\to C^*$.
Applying the shift functor $\s^n$ we now get a chain
map $C^*\to\s^n (C^*)$, and this corresponds to an element in
$\Ext_A^n(M^*,M^*)$.  One checks easily that this defines
an anti-isomorphism
$$
\mathcal D: \Ext_A(M,M) \xrightarrow{} \Ext_A(M^*,M^*)
$$
of graded rings.

\begin{definition} Let $M$ be an $A$-module of G-dimension zero.
We say that the central ring action from $H$ on
$D^b(A)$  \emph{commutes with dualization of $M$}, provided that the diagram
$$\xymatrix{
& H \ar[dl]_{\varphi_M} \ar[dr]^{\varphi_{M^*}} \\
\Ext_A(M,M) \ar[rr]^{\mathcal D} && \Ext_A(M^*,M^*) }$$ commutes,
that is, $\mathcal D(\varphi_M ( \eta )) = \varphi_{M^*} ( \eta )$ for every
homogeneous element $\eta \in H$.
\end{definition}

\section{Symmetric growth}\label{mainsection}

In this section, we prove the main result on symmetric growth in
totally acyclic minimal complexes of finitely generated free
modules. We start with the following proposition, which provides a
criterion for an extension to induce an eventually surjective chain
map on a minimal free resolution.

\begin{proposition}\label{surjectivechainmap}
Let $(A, \m, k)$ be a local ring, and let $M$ be a finitely
generated $A$-module with minimal free resolution $F$. Let
$\eta \in \Ext_A(M,M)$ be a homogeneous element of positive
degree, and suppose that it induces injective maps
$$\Ext_A^n(M,k) \xrightarrow{\cdot \eta} \Ext_A^{n+ |\eta|}(M,k)$$
for $n \gg 0$. Then any chain map on $F$ induced by $\eta$
is eventually surjective.
\end{proposition}

\begin{proof}
Let $\Omega_A^{|\eta|}(M) \xrightarrow{f_{\eta}} M$ be a map
representing the element $\eta$. Lifting this map along $F$
gives a chain map
$$
\xymatrix{
\cdots \ar[r] & F_{|\eta|+2} \ar[r] \ar[d]^{f_2} & F_{|\eta|+1}
\ar[r] \ar[d]^{f_1} & F_{|\eta|} \ar[r] \ar[d]^{f_0} &
\Omega_A^{|\eta|}(M) \ar[r] \ar[d]^{f_{\eta}} & 0 \\
\cdots \ar[r] & F_2 \ar[r] & F_1 \ar[r] & F_0 \ar[r] & M \ar[r] &
0}
$$
of degree $- |\eta|$ induced by $\eta$. Since the complex
$F$ is minimal, that is, the image of the differential $F_n
\to F_{n-1}$ is contained in $\m F_{n-1}$ for all $n$, the
differentials in the complex $\Hom_A( F, k)$ are all
trivial. We may therefore identify $\Ext_A^n(M,k)$ with
$\Hom_A(F_n,k)$ for all $n$. Moreover, under this identification,
for all $n$ the map
$$\Ext_A^n(M,k) \xrightarrow{\cdot \eta} \Ext_A^{n+ |\eta|}(M,k)$$
is the map
$$\Hom_A(F_n,k) \xrightarrow{f_n^*} \Hom_A(F_{n+
|\eta|},k)$$ induced by $f_n$. Applying $\Hom_A(-,k)$ to the right
exact sequence
$$F_{n+|\eta|} \xrightarrow{f_n} F_n \to \Coker f_n \to 0$$
induces the left exact sequence
$$0 \to \Hom_A( \Coker f_n, k) \to \Hom_A(F_n,k) \xrightarrow{f_n^*} \Hom_A(F_{n+
|\eta|},k).$$ By assumption, the map $f_n^*$ is injective for large
$n$, and so $\Hom_A( \Coker f_n, k)$ must vanish for $n \gg 0$.
However, the group $\Hom_A(X,k)$ is nonzero whenever $X$ is a
nonzero finitely generated $A$-module. Consequently, the map $f_n$
must be surjective for $n \gg 0$.
\end{proof}

The next result shows that when the cohomology of a module is
finitely generated, then its (positive) complexity is finite and its
Poincar{\'e} series is rational.

\begin{lemma}\label{finitecx}
Let $(A, \m, k)$ be a local ring, and let $M$ be a finitely
generated $A$-module.
\begin{enumerate}
\item[(i)] Suppose that $\Ext_A(M, k)$ belongs to $\Noethfl H$ for some
graded-commutative ring $H$ acting centrally on $D^b(A)$. Then
$\cx^{+}_A M$ is finite and the Poincar{\'e} series $P^{+}_A(M,t)$
is rational. Moreover, $\cx^{+}_A M$ equals the order of the pole of
$P^{+}_A(M,t)$ at $t=1$.
\item[(ii)] Suppose that $M$ has G-dimension zero and that
$\Ext_A(M^*, k)$ belongs to $\Noethfl H$ for some
graded-commutative graded ring $H$ acting centrally on $D^b(A)$.
Then $\cx^{-}_A M$ is finite and the Poincar{\'e} series
$P^{-}_A(M,t)$ is rational. Moreover, $\cx^{-}_A M$ equals the order
of the pole of $P^{-}_A(M,t)$ at $t=1$.
\end{enumerate}
\end{lemma}

\begin{proof}
It suffices to prove (i); claim (ii) follows from (i) and the fact
that $\beta_n(M^*) = \beta_{-(n+1)}(M)$ for all $n \in \mathbb{Z}$.

Since the scalar action from $H_0$ on $\Ext_A^n(M,k)$ factors
through $\Hom_A(k,k)$, the positive complexity of $M$ is the
complexity of the sequence $\{ \ell_{H_0} \Ext_A^n(M,k)
\}_{n=0}^{\infty}$. Therefore, by \cite[Lemma 2.6]{BIKO}, the
positive complexity of $M$ is finite.

Let $n_0$ be an integer such that the $H$-module $\Ext_A^{\ge
n_0}(M,k)$ is Noetherian and $\ell_{H_0} \left ( \Ext_A^n(M,k)
\right ) < \infty$ for each $n \ge n_0$, and denote the ideal
$\Ann_H \Ext_A^{\ge n_0}(M,k)$ in $H$ by $I$. By \cite[Remark
2.1]{BIKO}, the quotient ring $H/I$ is Noetherian, and its degree
zero part $(H/I)_0$ is Artin. The rationality of $P^{+}_A(M,t)$
now follows from the Hilbert-Serre Theorem (cf.\ \cite[Theorem
11.1]{AtiyahMacdonald}). The last part is a standard result on
Poincar{\'e} series (cf.\ \cite[Proposition 5.3.2]{Benson}).
\end{proof}

In order to prove the main result, we also need the following
elementary lemma. It shows that the category of modules of
G-dimension zero, and the category of modules whose cohomology is
finitely generated, are closed under extensions.

\begin{lemma}\label{exact sequence}
Let $(A, \m, k)$ be a local ring, and let
$$0 \to L \to M \to N \to 0$$
be an exact sequence of finitely generated $A$-modules.
\begin{enumerate}
\item[(i)] If $L$ and $N$ are both of G-dimension zero, then so is
$M$.
\item[(ii)] If $\Ext_A(L \oplus N,k)$ belongs to $\Noethfl H$
for some graded-commutative ring $H$ acting centrally on $D^b(A)$,
then so does $\Ext_A(M,k)$.
\end{enumerate}
\end{lemma}

\begin{proof}
For a proof of (i), see for example \cite[Lemma
1.1.10]{Christensen}. As for (ii), note that the exact sequence
induces an exact sequence
$$
\Ext_A(N,k) \to \Ext_A(M,k) \to \Ext_A(L,k)
$$
of graded $H$-modules. It follows that the middle term must be
eventually Noetherian of piecewise finite length.
\end{proof}

We now prove the main theorem. It shows that a totally acyclic
minimal complex of finitely generated free modules has symmetric
growth, provided its cohomology is finitely generated, and the
central action from $H$ commutes with dualization.

\begin{theorem}\label{main}
Let $(A, \m, k)$ be a local ring, and $C$ be a totally
acyclic minimal complex of finitely generated free $A$-modules.
Suppose that $\Ext_A(M_C \oplus M^*_C, k)$
belongs to $\Noethfl H$ for some graded-commutative ring $H$ acting
centrally on $D^b(A)$ and such that its action commutes with dualization
of $M$. Then $C$ has symmetric growth.
\end{theorem}

\begin{proof}
Denote the module $M_C$ by $M$. We prove the result by
induction on the positive complexity of $M$, which is finite by
Lemma \ref{finitecx}. If $\cx_A^{+} M =0$, then $M$ has finite
projective dimension and is therefore free. Thus, in this case, the
complex $C$ is bounded, and the result trivially follows.

Next, suppose that $\cx_A^{+}M$ is nonzero. By \cite[Lemma
2.5]{BIKO}, there exists a homogeneous element $\eta \in H$, of
positive degree, inducing injective maps
\begin{eqnarray*}
\Ext_A^n(M,k) & \xrightarrow{\cdot \eta} &
\Ext_A^{n+ |\eta|}(M,k) \\
\Ext_A^n(M^*,k) & \xrightarrow{\cdot \eta} & \Ext_A^{n+
|\eta|}(M^*,k)
\end{eqnarray*}
for $n \gg 0$. Choose maps $\Omega_A^{|\eta|}(M)
\xrightarrow{f_{\eta}} M$ and $\Omega_A^{|\eta|}(M^*)
\xrightarrow{g_{\eta}} M^*$ representing the elements $\varphi_M(
\eta )$ and $\varphi_{M^*}( \eta )$ in $\Ext_A(M,M)$ and
$\Ext_A(M^*,M^*)$, and note that the complexes $C_{\ge
0}$ and $(C_{\le -1} )^*$ are minimal free resolutions of
$M$ and $M^*$, respectively. By Proposition
\ref{surjectivechainmap}, the maps $f_{\eta}$ and $g_{\eta}$ induce
eventually surjective chain maps
$$
\xymatrix{
\cdots \ar[r] & C_{|\eta|+2} \ar[r] \ar[d]^{f_2} & C_{|\eta|+1}
\ar[r] \ar[d]^{f_1} & C_{|\eta|} \ar[r] \ar[d]^{f_0} &
\Omega_A^{|\eta|}(M) \ar[r] \ar[d]^{f_{\eta}} & 0 \\
\cdots \ar[r] & C_2 \ar[r] & C_1 \ar[r] & C_0 \ar[r] & M \ar[r] & 0
\\
\cdots \ar[r] & C_{-(|\eta|+3)}^* \ar[r] \ar[d]^{g_2} &
C_{-(|\eta|+2)}^* \ar[r] \ar[d]^{g_1} & C_{-(|\eta|+1)}^* \ar[r]
\ar[d]^{g_0} &
\Omega_A^{|\eta|}(M^*) \ar[r] \ar[d]^{g_{\eta}} & 0 \\
\cdots \ar[r] & C_{-3}^* \ar[r] & C_{-2}^* \ar[r] & C_{-1}^* \ar[r]
& M^* \ar[r] & 0}
$$
on $C_{\ge 0}$ and $(C_{\le-1} )^*$, respectively. Since $g_n$ is
split surjective for $n \gg
0$, when dualizing the lower diagram we obtain a chain map
$$\xymatrix{
0 \ar[r] & M \ar[r] \ar[d]^{g_{\eta}^*} & C_{-1} \ar[r]
\ar[d]^{g_0^*} & C_{-2} \ar[r] \ar[d]^{g_1^*} & C_{-3} \ar[r]
\ar[d]^{g_2^*} & \cdots \\
0 \ar[r] & \Omega_A^{-|\eta|}(M) \ar[r] & C_{-(|\eta|+1)} \ar[r] &
C_{-(|\eta|+2)} \ar[r] & C_{-(|\eta|+3)} \ar[r] & \cdots}$$ which is
eventually split injective.

Choose a short exact sequence
$$0 \to M \to K \to \Omega_A^{|\eta|-1}(M) \to 0$$
representing the element $\varphi_M( \eta )$ in $\Ext_A(M,M)$. By
Lemma \ref{exact sequence}, the module $K$ also has G-dimension
zero, and both $\Ext_A(K,k)$ and $\Ext_A(K^*,k)$ belong to
$\Noethfl H$. From the sequence we obtain a long exact sequence
$$\cdots \to \Tateext_A^{n}(M,k) \xrightarrow{\partial_n}
\Tateext_A^{n+ |\eta|}(M,k) \to \Tateext_A^{n+1}(K,k) \to
\Tateext_A^{n+1}(M,k) \xrightarrow{\partial_{n+1}} \cdots$$ in Tate
cohomology. For $n \ge 0$, the connecting homomorphism $\partial_n$
is induced by $f_n$, and so since $f_n$ is surjective for large $n$,
we see that $\partial_n$ is injective for $n \gg 0$. Moreover, the
central action from $H$ on $D^b(A)$ commutes with dualizations,
that is, $\mathcal D(\varphi_M( \eta )) = \varphi_{M^*}( \eta )$, hence
$\partial_n$ is induced by $g_{-(n+ |\eta|+1)}^*$ for $n \ll 0$. The
map $g_n^*$ is split injective for large $n$, and consequently
$\partial_n$ is surjective for $n \ll 0$.

Choose a number $n_0$ with the property that $\partial_n$ is
injective for $n \ge n_0$ and surjective for $n \le -n_0$. Then
the sequences
$$0 \to \Tateext_A^n(M,k) \to \Tateext_A^{n+ |\eta|}(M,k) \to \Tateext_A^{n+1}(K,k)
\to 0$$
$$0 \to \Tateext_A^{-n}(K,k) \to \Tateext_A^{-n}(M,k) \to \Tateext_A^{|\eta|-n}(M,k)
\to0$$ are exact for $n \ge n_0$, giving equalities
\begin{eqnarray*}
\beta_{n+1}(K) & = & \beta_{n+ |\eta|}(M) - \beta_n (M) \\
\beta_{-n}(K) & = & \beta_{|\eta|-n}(M) - \beta_{-n} (M)
\end{eqnarray*}
of Betti numbers. Computing Poincar{\'e} series, we obtain
\begin{eqnarray*}
P^{+}_A(K,t) & = & \frac{(1-t^{|\eta|})P^{+}_A(M,t)}{t^{|\eta|-1}} +f(t)\\
P^{-}_A(K,t) & = & (t^{|\eta|}-1)P^{-}_A(M,t) +g(t)
\end{eqnarray*}
for some polynomials $f(t),g(t) \in \mathbb{Z}[t]$. Consequently,
the order of the pole of $P^{+}_A(K,t)$ at $t=1$ is one less than
that of $P^{+}_A(M,t)$, whereas the pole of $P^{-}_A(K,t)$ is one
less than that of $P^{-}_A(M,t)$. Therefore, by Lemma
\ref{finitecx}, the positive complexity of $K$ is $\cx_A^{+}M -1$,
the negative complexity of $K$ is $\cx_A^{-}M -1$,
and by induction we obtain
$$\cx_A^{+}M = \cx_A^{+}K +1 = \cx_A^{-} K+1 = \cx_A^{-}M.$$
This shows that $\mathbf{C}$ has symmetric growth.
\end{proof}

We include the following equivalent version of the theorem as a
corollary. It follows from the elementary fact that for a module of
G-dimension zero, the negative complexity equals the positive
complexity of its dual.

\begin{corollary}\label{equivalent}
Let $(A, \m, k)$ be a local ring, and let $M$ be a finitely
generated $A$-module of G-dimension zero. If $\Ext_A(M \oplus M^*,
k)$ belongs to $\Noethfl H$ for some graded-commutative ring $H$
acting centrally on $D^b(A)$ and such that its action commutes
with dualization of $M$, then $\cx_A^{+} M = \cx_A^{-} M =
\cx_A^{+} M^* = \cx_A^{-} M^*$.
\end{corollary}

\section{Complete intersection dimension zero}\label{CI-dim}

Assume that $A=B/(x_1,\dots,x_c)$, where $B$ is a local ring
and $x_1,\dots,x_c$ is a $B$-regular sequence.
Then there exists
a polynomial ring
$$
H = A[\chi_1, \dots, \chi_c ]
$$
acting centrally on $D^b(A)$, with each \emph{cohomology operator}
$\chi_i$ of degree two. For purposes below, we recall the definition
of the elements $\varphi_M(\chi_i)\in \Ext_A^2(M,M)$, and their action
on $\Ext_A(M,N)$ and $\Ext_A(N,M)$.
(There are actually several definitions for these elements,
but they all agree up to sign; see \cite{AvramovSun}.  The one we give is
from \cite{Eisenbud}).

Let $(C,d)$ be a complex of free $A$-modules.  We lift $C$ to a sequence
of maps $(\widetilde C,\widetilde d)$ of free $B$-modules. Since
$\widetilde d^2\equiv 0$ modulo $(x_1,\dots,x_c)$ we can decompose
$\widetilde d^2$ as
$$
\widetilde d^2=\sum_{i=1}^c x_i\widetilde t_i
$$
for some family $(\widetilde t_i)_{i=1}^c$ of degree $-2$
endomorphisms of the graded $B$-module $\widetilde C$. Then
$t_i=\widetilde t_i\otimes_B A$ become degree $-2$ chain
maps on the complex $C$ which are well-defined and commute
up to homotopy (see \cite{Eisenbud}).  The chain maps $t_i$ on a free
resolution $C$ of $M$ then define
the elements $\varphi_M(\chi_i)\in\Ext^2_A(M,M)$.  The action of
$\varphi_M(\chi_i)$ on $\Ext_A(M,N)$ and
$\Ext_A(N,M)$ is thus determined by composition of chain maps.
If $M=M_C$ for a totally acyclic complex $C$
of free $A$-modules, the element $\varphi_M(\chi_i)\in\Ext_A^2(M,M)$ is
determined by the chain map $t_i:C\to\s^2 C$.

Recall that a \emph{quasi-deformation} of local rings is a diagram
of local homomorphisms $A\to A' \leftarrow B$ such that the map $A
\to A'$ is flat and $A'\leftarrow B$ is onto with kernel generated
by a $B$-regular sequence contained in the maximal ideal of
$B$.  A finitely generated module $M$ over a local ring $A$ is
said to have \emph{finite complete intersection dimension},
denoted $\CIdim_A M<\infty$, if there exists a quasi-deformation
of local rings $A\to A' \leftarrow B$ such that $\pd_B M \otimes_A
A'<\infty$ \cite{AvramovGasharovPeeva}. In \emph{loc. cit.} it is
shown that if $M$ has finite CI-dimension, then $M$ has finite
G-dimension, and when both are finite, that
 $$
 \CIdim_AM=\Gdim_AM=\depth_A A -\depth_A M
 $$

\begin{lemma} \label{commutes}
Let $A=B/(x_1,\dots,x_c)$ where $(B,\n,k)$ is a local
ring, and $x_1,\dots,x_c$ is a $B$-regular sequence contained in
$\n$. Suppose that $M$ is an $A$-module with $\depth_AM=\depth_AA$,
and $\pd_BM<\infty$.
Then the central ring action of the polynomial ring of cohomology operators
$H=A[\chi_1,\dots,\chi_c]$  on $D^b(A)$ commutes with dualization of $M$.
\end{lemma}

\begin{proof} The hypotheses show that the CI-dimension of
$M$ is zero, and therefore $M$ has G-dimension zero.  Let
$C$ be a totally acyclic complex of free modules with $M=M_C$.
Choose $\chi=\chi_i\in H$. It suffices to prove that
$\mathcal D(\varphi_M(\chi))=\varphi_{M^*}(\chi)$.
The element $\varphi_M(\chi)\in\Ext_A(M,M)$ is determined by the
chain map $t:C \to \s^2 C$, as described above.
Therefore $\mathcal D(\varphi_M(\chi))$
corresponds to the chain map $\s^2t^*:C^*\to\s^2 C^*$.
On the other hand, if one uses the lifting
$\Hom_B(\widetilde C,B)$ of $C^*$, and the
factorization of $(\widetilde\partial^*)^2$ dual to that
of $\widetilde\partial^2$ we see that the chain map
$\s^2t^*:C^*\to\s^2 C^*$ defines the element
$\varphi_{M^*}(\chi)\in\Ext_A(M^*,M^*)$, and this is what we wanted to show.
\end{proof}

We have the following corollary of Theorem \ref{main}, in the case
of CI-dimension zero.

\begin{corollary} Let $(A, \m, k)$ be a local ring, and $C$ be a totally
acyclic minimal complex of finitely generated free $A$-modules.
If $M_C\oplus (M_C)^*$ has CI-dimension zero, then $C$ has symmetric
growth.
\end{corollary}

\begin{proof} Set $M=M_C$, and let $A\to A'\leftarrow B$
be a quasi-deformation such that $\pd_B((M\oplus M^*)\otimes_A
A')<\infty$. The complex $C$ has symmetric growth if and only if
the complex $C'=C\otimes_AA'$ does, the latter complex being a
totally acyclic minimal complex of finitely generated free
$A'$-modules. Moreover, $\Hom_A(M,A)\otimes_A A'\cong
\Hom_{A'}(M\otimes_A A',A')$, and so $\pd_B \left((M\otimes_A A')
\oplus \Hom_{A'}(M\otimes_A A',A')\right)<\infty$. Also $\depth_AA
=\depth_{A'}A'$, and $\depth_A M=\depth_{A'} M\otimes_A A'$.
Changing notation, we can therefore assume that
$A=B/(x_1,\dots,x_c)$, where $B$ is a local ring and
$x_1\dots,x_c$ is a $B$-regular sequence contained in the 
maximal ideal of $B$, that $C$ is a totally acyclic minimal
complex of finitely generated free $A$-modules, and that $M=M_C$
and $M^*=\Hom_A(M,A)$ are $A$-modules such that $\depth_A
M=\depth_A A$ and $\pd_B M\oplus M^*<\infty$. By Lemma
\ref{commutes} we have that the central ring action of the
polynomial ring of cohomology operators $H=A[\chi_1,\dots,\chi_c]$
on $D^b(A)$ commutes with dualization of $M$.  Therefore by
Theorem \ref{main} it suffices to know that $\Ext_A(M \oplus M^*,
k)$ belongs to $\Noethfl H$, but this is implied by the main
result of \cite{Gulliksen} (see also \cite{Eisenbud} and
\cite{Avramov}).
\end{proof}

When $A$ is a complete intersection, then every finitely generated
$A$-module has finite CI-dimension.
Therefore, in this case, every totally acyclic minimal
complex of free modules has symmetric growth. This follows already from
\cite[Theorem 5.6]{AvramovBuchweitz}, but because of our alternative
proof, we include it here as a corollary.

\begin{corollary}\label{ci}
Let $(A, \m, k)$ be a local complete intersection. Then every
totally acyclic minimal complex of finitely generated free
$A$-modules has symmetric growth.
\end{corollary}

The following construction indicates that finitely generated modules
$M$ over a non-complete intersection ring such that
$M\oplus M^*$ has CI-dimension zero are not at all rare.

\begin{construction} Let $(A_1,\m_1,k)$ and $(A_2,\m_2,k)$
be local rings which both contain their common residue
field $k$. Assume that $A_1=B_1/(x_1,\dots,x_c)$ where
$(B_1,\n_1,k)$ is a regular local ring  and $x_1,\dots,x_c$
is a $B_1$-regular sequence contained in $\n_1$ (so that
$A_1$ is a complete intersection).
Let $\m=\m_1\otimes_kA_2 + A_1\otimes_k\m_2$, a maximal ideal
of $A_1\otimes_kA_2$, and set $A=(A_1\otimes_kA_2)_{\m}$.  Then
$(A,\m,k)$ is a local ring with residue field $k$ also,
and which is not a complete intersection if $A_2$ is not a
complete intersection.

Now let $M_1$ be any maximal Cohen-Macaulay $A_1$-module, equivalently
any $A_1$-module with $\CIdim_{A_1}M_1=0$.  Then
$\Hom_{A_1}(M_1,A_1)$ is also of CI-dimension zero. It is easy to
see that $M=(M_1\otimes_k A_2)_{\m}$ and
$$
M^*=\Hom_A(M,A)\cong(\Hom_{A_1}(M_1,A_1)\otimes_k A_2)_{\m}
$$ 
are finitely generated $A$-modules of CI-dimension zero.
\end{construction}

We end with some questions for further study.

\begin{question} Let $M$ be a finitely generated module over a local
ring $(A,\m,k)$.  If $M$ has finite CI-dimension, then does
$M^*=\Hom_A(M,A)$ have finite CI-dimension as well?
In  particular, if $A=B/(x)$ for
$(B,\n,k)$ a local ring and $x$ a non-zerodivisor in $\n$,
then does $\pd_B M<\infty$ imply that $\pd_B M^*<\infty$ as well?
\end{question}

\begin{question} Do there exist graded commutative rings $H$
acting centrally on $D^b(A)$ such that the action fails
to commute with dualization of a finitely generated $A$-module
$M$?
\end{question}

\section*{Acknowledgements}

This work was done while the second author was visiting Trondheim,
Norway, December-January 2008-9.  He thanks the Algebra Group at the
Institutt for Matematiske Fag, NTNU, for their hospitality and
generous support. The first author was supported by NFR Storforsk
grant no.\ 167130.

\end{document}